\documentclass[12pt]{article}
\usepackage{amsmath,amssymb}
\usepackage{amsthm}

\def\E{\mathbf{E}}

\def\P{\mathbf{P}}
\def\R{\mathbf{R}}
\def\Z{\mathbf{Z}}

\def\1{\mathbf{1}}

\def\sgn{\rm{sgn}}

\def\al{\alpha}

\def\pa{\partial}
\def\ep{\epsilon}

\newtheorem{prop}{Proposition}[section]
\newtheorem{theorem}{Theorem}[section]

\newtheorem{corollary}{Corollary}

\newcommand{\om}{\omega}

\begin{document}
\title{Stochastic monotonicity and duality for one-dimensional Markov processes (revised)}
\author{Vassili N. Kolokoltsov\thanks{Department of Statistics, University of Warwick,
 Coventry CV4 7AL UK,
   Email: v.kolokoltsov@warwick.ac.uk}}

\maketitle

\begin{abstract}
The theory of monotonicity and duality is developed for general one-dimensional Feller processes, extending the approach from
\cite{Ko03}. Moreover it is shown that local monotonicity conditions (conditions on the L\'evy kernel)
are sufficient to prove the well-posedness of the corresponding Markov semigroup and process, including unbounded
coefficients and processes on the half-line.
\end{abstract}


\section{Introduction}

A Markov process $X_t$ in $\R$ is called {\it stochastically monotone}\index{stochastically monotone}
if the function $\P(X_t^x \ge y)$ (as usual, $x$ stands for the initial point here)
 is nondecreasing in $x$ for any $y\in \R, t\in \R_+$,
 or, equivalently (by linearity and approximation), if the corresponding Markov semigroup preserves the set of non-decreasing functions.
  A Markov process $Y_t$ in $\R$
is called {\it dual}\index{dual process} to $X_t$ if
\begin{equation}
\label{eqdefonedimdual}
\P(Y_t^y \le x)=\P(X_t^x \ge y)
\end{equation}
for all $t>0$, $x,y \in \R$. If a dual Markov process exists it is obviously unique.

Stochastic monotonicity for Markov chains is well studied and applied for the analysis of many practical models, see e.g.
\cite{And}, \cite{Con}, \cite{Da}, \cite{Maa}.
Stochastic monotonicity and the related duality are also well studied for diffusions
(see \cite{Kal} and \cite{ChWa})
and jump-type Markov processes (see \cite{And}, \cite{ChZh}, \cite {Chen}), \cite{Zh}).
In \cite{SaTa} the monotonicity for stable processes was analyzed.
For general Markov processes the analysis of stochastic monotonicity and related duality was initiated in
\cite{Ko03} devoted to the case of one dimensional processes with polynomial coefficients 
(note some nasty typos in the expression of the dual generator in \cite{Ko03}).
This was related to interacting particle models (see also \cite{La}), and related Markov models
in financial mathematics.
In \cite{JMW} the theory of monotonicity was extended to multidimensional processes of L\'evy-Khintchine type
with L\'evy measures having a finite first internal moment, i.e. with $\int_{|y|<1} \nu (x,dy)<\infty$ (slightly more general in fact).

In this note we first extend the theory of monotonicity and duality to arbitrary one-dimensional Feller processes (Sections 2 and 3), following approach from \cite{Ko03}. We shall give a criterion
of stochastic monotonicity in terms of the generator of $X_t$ and, under additional regularity assumptions,
the explicit formula for the dual generator.
Here our approach is based on the discretization and eventually
on the theory of stochastic monotonicity for Markov chains.

In the second part of the paper (Sections 4,5) we use an alternative approach to the analysis of monotonicity, 
adapting in particular the method used in \cite{Ko10book} for the generators of order at most one.
 Most importantly we show that local monotonicity conditions (conditions on the L\'evy kernel)
are sufficient to prove the well-posedness of the corresponding Markov semigroup and process,
 thus contributing to the important problem of building a process from a given pre-generator (see e.g. \cite{Ja}, \cite{Ko04},
\cite{Ko10}). Stochastically monotone processes
on the half-line are finally constructed.

Most of the results given are extendable to arbitrary dimensions, but the exposition of one-dimensional theory as a first step 
seems to be in order, not least because its relevance to option pricing, see \cite{MiPi}. 

To conclude the introduction, let me thank professor Mu Fa Chen for bringing to my attention some relevant recent
 publications of the Chinese school.
   
\section{Monotonicity via discrete approximations}

Let us recall shortly the
theory of stochastic monotonicity for Markov chains, following \cite{And}.
Recall first that an {\it infinitesimally stochastic matrix}\index{infinitesimally stochastic matrix}
or $Q$-{\it matrix}\index{$Q$-matrix}
 $Q=(Q_{mn})$ with $m,n \in \Z$ is such a matrix that $Q_{nm} \ge 0$ for $m\neq n$ and
 \begin{equation}
\label{eqdefQmatrixcond}
 Q_{nn}=-\sum_{m\neq n} Q_{nm}
 \end{equation}
 for all $n$. To any such matrix there corresponds a Markov process (generally not unique) with the generator
 given by the matrix $Q$ (which we shall denote by the same letter):
 \[
 (Qf)_n=\sum_m Q_{nm}f_m.
\]
Taking into account the properties of $Q$, one can rewrite it in other two useful forms:
\[
 (Qf)_n=\sum_{m\neq n} Q_{nm}(f_m-f_n)=\sum_m Q_{nm}(f_m-f_n).
\]
If the intensity of jumps, specified $Q$ is uniformly bounded, that is
\[
\sup_n |Q_{nn}| <\infty,
\]
the corresponding Markov process is unique and conservative, the latter meaning that the
corresponding semigroup preserves constants.

A $Q$-matrix is called
{\it stochastically monotone}\index{infinitesimally stochastic matrix!stochastically monotone}
if
\begin{equation}
\label{eqdefstochmonotmatrix}
\sum_{j \ge l} Q_{nj} \le \sum_{j \ge l} Q_{n+1,j} \quad \forall \, l\neq n+1
\end{equation}
(we separate indices by commas, like $Q_{n,m}$, if needed for clearness).
The key discrete result states (proof in \cite{And}) that if $Q$ is stochastically
 monotone, then the corresponding Markov chain $X_t$ is stochastically monotone in the sense that
 $\P(X_t^n \ge m)$ is nondecreasing in $n$ for any $m, t$ and the dual Markov chain $Y_t$
  (satisfying \eqref{eqdefonedimdual} with integer $x,y$) has the $Q$-matrix
\begin{equation}
\label{eqdualQmatrix}
\tilde Q_{nj}=\sum_{l=n}^{\infty} (Q_{jl}-Q_{j-1,l}).
\end{equation}

It turns out that the monotonicity condition becomes more transparent if written in terms of the matrix $\om=(\om_{nm})$,
which is connected with the $Q$-matrix by the equation $\om_{nm}=Q_{n \, n+m}$. Thus the entries $\om_{nm}$ define the probabilities
of jumps to the right ($m>0$) and to the left ($m<0$) of $n$. In fact, condition \eqref{eqdefstochmonotmatrix} takes form
\begin{equation}
\label{eqdefstochmonotmatrix1}
\sum_{m \ge l-n} \om_{nm} \le \sum_{m \ge l-n-1} \om_{n+1,m} \quad \forall \, l\neq n+1.
\end{equation}
and this condition is equivalent to two separate conditions on the right and the left jumps:
\begin{equation}
\label{eqdefstochmonotmatrix2}
\sum_{m \ge k} \om_{nm} \le \om_{n+1, k-1}+\sum_{m \ge k} \om_{n+1,m} \quad \forall \, k\ge 2,
\end{equation}
\begin{equation}
\label{eqdefstochmonotmatrix3}
\om_{n,-k+1}+\sum_{m \ge k} \om_{n, -m} \ge \sum_{m \ge k} \om_{n+1,-m} \quad \forall \, k\ge 2.
\end{equation}

{\bf Remark.} {\it It is straightforward to see that \eqref{eqdefstochmonotmatrix2} is equivalent to \eqref{eqdefstochmonotmatrix1} for $l\ge n+2$.
equation \eqref{eqdefQmatrixcond} implies that
\eqref{eqdefstochmonotmatrix3} is equivalent to \eqref{eqdefstochmonotmatrix} for $l\le n$.}

It is worth noting that \eqref{eqdefstochmonotmatrix2}, \eqref{eqdefstochmonotmatrix3} are satisfied if $\om_{n,1}, \om_{n,-1}$
are arbitrary (non-negative) and other coefficients satisfy simpler inequalities
\begin{equation}
\label{eqdefstochmonotmatrix4}
\sum_{m \ge k} \om_{nm} \le \sum_{m \ge k} \om_{n+1,m}, \quad
\sum_{m \ge k} \om_{n, -m} \ge \sum_{m \ge k} \om_{n+1,-m} \quad \forall \, k\ge 2.
\end{equation}
Finally equation \eqref{eqdualQmatrix} rewrites as
\[
\tilde Q_{n,n+i}=\sum_{l=-i}^{\infty} (Q_{n+i,n+i+l}-Q_{n+i-1,n+i+l})
\]
and hence in terms of $\tilde \om_{nm}=\tilde Q_{n,n+m}$ as
\begin{equation}
\label{eqdualQmatrix1}
\tilde \om_{ni}=\sum_{l=-i}^{\infty} (\om_{n+i,l}-\om_{n+i-1,l+1}).
\end{equation}
In particular, if $\om_{nm}$ do not vanish only for $|m|\le 1$, the same holds for $\tilde \om$ and
 \begin{equation}
\label{eqdualQmatrix2}
\tilde \om_{n1}=\om_{n,-1}, \quad \tilde \om_{n,-1}=\om_{n-1,1}.
\end{equation}
Moreover, by duality, right jumps turn to the left jumps and vice versa, i.e. if $\om_{n,-i}=0$ for all $i>0$, then
$\tilde \om_{ni}=0$ for $i>0$ and
  \begin{equation}
\label{eqdualQmatrix3}
\tilde \om_{n, -i}=\om_{n-i,i}+\sum_{l=i+1}^{\infty} (\om_{n-i,l}-\om_{n-i-1,l}) \quad i>0;
\end{equation}
and if
$\om_{n,i}=0$ for all $i>0$, then
$\tilde \om_{n,-i}=0$ for $i>0$ and
  \begin{equation}
\label{eqdualQmatrix4}
\tilde \om_{n, i}=\om_{n+i-1,-i}+\sum_{l=i+1}^{\infty} (\om_{n+i-1,-l}-\om_{n+i,-l}) \quad i>0.
\end{equation}

The following is the main result of this short paper.

\begin{theorem}
Let $X_t$ be a Feller process in $\R$ with the generator of the usual L\'evy-Kchintchine form
 \begin{equation}
\label{eqgenLKonedim}
Lf(x)=\frac{1}{2}G(x) f''(x)+b(x) f'(x)
+\int (f(x+y)-f(x)-f'(x)y\1_{B_1}(y)) \nu (x, dy)
\end{equation}
with continuous $G,b, \nu$, and let the space $C_c^2(\R)$ be a core.
For simplicity assume also (though this is not very essential) that the coefficients are bounded, that is
\[
\sup_x \left(G(x)+|b(x)|+\int (1\land y^2) \nu (x, dy)\right) <\infty.
\]
 If the L\'evy measures $\nu$ are such
that for any $a>0$ the functions
 \begin{equation}
\label{eqcondmonotonlevymeasure}
\int_a^{\infty} \nu (x,dy), \quad \int_{-\infty}^{-a} \nu (x,dy)
\end{equation}
are non-decreasing and non-increasing respectively, as functions of $x$,
then the process $X_t$ is stochastically monotone.
 Moreover, the dual Markov process exists.
 \end{theorem}

\begin{proof}
Let $h>0$ and set
\[
L_hf(x)=G(x)\frac{f(x+h)+f(x-h)-2f(x)}{2h^2}+|b(x)| \frac{f(x+h \, \sgn (b(x)))-f(x)}{h}
\]
\[
+\sum_{m=1}^{\infty} [f(x+mh)-f(x)+\frac{f(x-h)-f(x)}{h}mh\1_{B_1}(mh)] \nu (x, [mh, mh+ h))
\]
\begin{equation}
\label{eqgenLKonedimdiscrappr}
+\sum_{m=1}^{\infty} [f(x-mh)-f(x)+\frac{f(x+h)-f(x)}{h}mh\1_{B_1}(mh)] \nu (x, (mh-h, mh]).
\end{equation}
These operators approximate $L$ on $C_c^2(\R)$ for $h\to 0$. Since this space is a core for $L$, the corresponding semigroups converge.
But by the above mentioned result for Markov chains, the processes $X_{t,h}$ generated by $L_h$ are stochastically monotone.
Consequently the same holds for the process $X_t$ generated by $L$. Again by the properties of Markov chains, the dual processes
$Y_{t,h}$ to $X_{t,h}$ are well defined. Their transition probabilities converge, because they are expressed in terms of the converging transition
probabilities of $X_{t,h}$. The limiting Markov process $Y_t$ is dual to $X_t$.
\end{proof}

\section{Dual generators}

Under some regularity assumptions we can write explicitly the generator of the dual process. To simplify formulas,
we shall do it only
for L\'evy measures supported on $\R_+$ (the case of measures supported on $\R_-$ is symmetric and is done using
equation \eqref{eqdualQmatrix4} instead of \eqref{eqdualQmatrix3} used below).

\begin{prop}
Under the assumptions of the above Theorem
suppose additionally that the L\'evy measures are supported on $\R_+$ and either (i) $\nu (x, dy)=\nu (x,y) \, dy$ with $\nu(x,y)$
differentiable in $x$, or (ii) $\nu (x, dy)=a(x) \nu (dy)$ with a certain L\'evy measure $\nu$
and a continuously differentiable function $a$ (decomposable generator case).
Then the generator of the dual Markov process acts by
\[
Lf(x)=\frac{1}{2}G(x) f''(x)-[\frac{1}{2} G'(x)+b(x)] f'(x)
\]
\begin{equation}
\label{eqgenLKonedimdualtochmonot}
+\int_0^{\infty} [f(x-y)-f(x)+f'(x)\1_{B_1}(y)] \tilde \nu (x, dy)
+f'(x) \int_0^1 y (\nu -\tilde \nu) (x,dy)
\end{equation}
on $C_c^2(\R)$, where
\[
\tilde \nu (x,dy)=[\nu (x-y,y)+\frac{\pa}{\pa x} \int_y^{\infty} \nu (x-y,z) \, dz ] dy
\]
in case (i) and
\[
\tilde \nu (x,dy)= a(x-y) \nu (dy) +a' (x-y) \int_y^{\infty} \nu (dz) \, dy
\]
in case (ii).
\end{prop}

\begin{proof}
By linearity one can calculate the dual generator separately for diffusive, drift and integral parts of $L$.
By \eqref{eqdualQmatrix2} the dual generator corresponding to the first term in \eqref{eqgenLKonedimdiscrappr} has the form
\[
(2h^2)^{-1}[ G(x)f(x+h)+G(x-h)f(x-h)-(G(x)+G(x-h))f(x)],
\]
which converges to
\[
\frac{1}{2} [G(x)f''(x)-G'(x)f'(x)],
\]
as $h\to 0$.
Similarly analyzing the drift part yields the first two terms in \eqref{eqgenLKonedimdualtochmonot}.

Next, from equation \eqref{eqdualQmatrix3}, it follows that the dual operator to the first sum in
of \eqref{eqgenLKonedimdiscrappr} equals
\[
\sum_{m=1}^{\infty} (f(x-mh)-f(x))\bigl[ \nu (x-mh, [mh, mh+ h))
\]
\[
+\sum_{l=m+1}^{\infty} \nu (x-mh, [lh, lh+ h))-\nu (x-mh-h, [lh, lh+ h))\bigr]
\]
\[
+ \frac{f(x+h)-f(x)}{h} \sum_{m=1}^{\infty} mh\1_{B_1}(mh) \nu (x, [mh, mh+ h)).
\]
In case (i) it rewrites as
\[
\sum_{m=1}^{\infty} (f(x-mh)-f(x))\left[ \int_{mh}^{mh+h}\nu (x-mh, y) \, dy
+\sum_{l=m+1}^{\infty} \int_{lh}^{lh+h}( \nu (x-mh,z)-\nu (x-mh-h,z))\, dz\right]
\]
\[
+\frac{f(x+h)-f(x)}{h} \sum_{m=1}^{\infty} mh\1_{B_1}(mh) \int_{mh}^{mh+h} \nu (x, y) \, dy,
\]
yielding the first formula for $\tilde \nu$. In case (ii), it rewrites as
\[
\sum_{m=1}^{\infty} (f(x-mh)-f(x))\left[a(x-mh)\nu ([mh, mh+h))
+\sum_{l=m+1}^{\infty} (a(x-mh)-a(x-mh-h)\nu ([lh,lh+h)) \right]
\]
\[
+\frac{f(x+h)-f(x)}{h} \sum_{m=1}^{\infty} mh\1_{B_1}(mh) a(x) \nu ([mh, mh+h)) \, dy,
\]
yielding the second one.

\end{proof}

\section{Well-posedness via monotonicity}

Apart from the definition of monotonicity, the remaining exposition is independent of the previous results.
Let us first describe our approach in the simplest situation.

\begin{theorem}
\label{thonedimstochmonwellposed}
Let
\begin{equation}
\label{eqgenLKonedimjum}
Lf(x)=\int (f(x+y)-f(x)-f'(x)y) \nu (x, dy)
\end{equation}
with a continuous L\'evy kernel $\nu$ such that
\[
\sup_x \int (|y|\land |y|^2) \nu (x, dy) <\infty,
\]
and the first two derivatives $\nu'(x,dy)$ and $\nu'(x,dy)$ of $\nu$ with respect to $x$ exist weakly and
define continuous signed L\'evy kernels such that
\[
\sup_x \int (|y|\land |y|^2) |\nu' (x, dy)| <\infty, \quad
\sup_x \int (|y|\land |y|^2) |\nu '' (x, dy)| <\infty.
\]
If for any $a>0$ the functions
 \begin{equation}
\label{eqcondmonotonlevymeasure}
\int_a^{\infty} \nu (x,dy), \quad \int_{-\infty}^{-a} \nu (x,dy)
\end{equation}
are non-decreasing and non-increasing respectively,
then $L$ generates a unique Feller semigroup with the generator given by
\eqref{eqgenLKonedimjum} on the subspace $C_{\infty}(\R)\cap C^2(\R)$.
The corresponding process is stochastically monotone.
\end{theorem}

\begin{proof}
We shall use the following two Taylor formulas:
\[
f(x+y)-f(x)-f'(x)y=\int_0^y (f'(x+z)-f'(x)) dz=\int_0^y (y-z)f''(x+z) dz,
\]
where of course  $\int_0^y=-\int_y^0$ for $y<0$.

Differentiating the equation $\dot f=Lf$ with respect to the spatial variable $x$ yields the following
equation for $g(x)=f'(x)$:
\[
\frac{d}{dt} g(x)= (L+K)g(x),
\]
where
\[
Kg(x)=\int (f(x+y)-f(x)-f'(x)y) \nu '(x, dy)
\]
\[
=\int_0^{\infty} \left(\int_0^y (g(x+z)-g(x)) dz\right) \nu' (x, dy)
-\int_{-\infty}^0 \left(\int_y^0 (g(x+z)-g(x)) dz\right) \nu' (x, dy)
\]
\[
=\int_0^{\infty} dz (g(x+z)-g(x)) \int_z^{\infty}  \nu' (x, dy)
-\int_0^{\infty} dz (g(x-z)-g(x)) \int_{-\infty}^{-z}  \nu' (x, dy).
\]
The main observation is that by the assumptions of the theorem both terms represent
conditionally positive operators of the L\'evy-Khintchine type.
Differentiating once more one gets for $v=g'=f''$ the equation
\[
\frac{d}{dt} v(x)= (L+2K)v(x)+\int (f(x+y)-f(x)-f'(x)y) \nu'' (x, dy),
\]
\begin{equation}
\label{eqgenLKonedimjumsecder}
=(L+2K)v(x)+\int _{-\infty}^{\infty}
\left[ \1_{|y|\le 1}\int_0^y (y-z)v(x+z) dz+\1_{|y|> 1}\int_0^y (g(x+z)-g(x)) dz\right] \nu '' (x,dy),
\end{equation}
and the last term represents a sum of a bounded operator applied to $v$ and a bounded curve whenever $g$ is bounded.

To make the rigorous analysis let us introduce the approximating operator $L_h$, $h>0$, as
\[
L_hf(x)=\int_{|y|>h} (f(x+y)-f(x)-f'(x)y) \nu (x, dy).
\]
 Then $L_h$ is the sum of the first order operator and a bounded operator in $C_{\infty}(\R)$
  (the latter is due to our assumptions on the moment of $\nu$). Hence it generates a conservative Feller semigroup $T_t^h$ for any $h>0$.
 By the form of the equations for $f'$, i.e. $\dot g=(L_h+K_h)g$ with bounded (in $C(\R)$) and
 conditionally positive $K_h$, and also $f''$, one concludes that this semigroup acts
by positive contractions on the derivatives $g=f'$ and by bounded operators on $v=f''$ uniformly in $h$.
Hence the spaces $C_{\infty}(\R)\cap C^1(\R)$ and $C_{\infty}(\R)\cap C^2(\R)$ are both invariant under $T_h^t$.
 Moreover, for any $f\in C_{\infty}(\R)\cap C^2(\R)$, the functions $T_h^tf$ belong to $C_{\infty}(\R)\cap C^2(\R)$
 with bounds uniform in $h\in (0,1]$ and $t\in [0,t_0]$ for any $t_0$.

 Therefore, writing
\[
(T_t^{h_1}-T_t^{h_2})f=\int_0^t
T_{t-s}^{h_2}(L_{h_1}-L_{h_2})T_s^{h_1}\, ds
\]
for arbitrary $h_1 >h_2$ and estimating
\[
|(L_{h_1}-L_{h_2})T_s^{h_1}f(x)|
 \le \int_{B_{h_1}}
 \|T_s^{h_1}f\|_{C_2(\R)}|y|^2 \nu (x,dy)=o(1) \|f\|_{C^2}(\R), \quad h_1\to 0,
 \]
yields
\begin{equation}
\label{eqshortdifference}
 \|(T_t^{h_1}-T_t^{h_2})f\|=o(1) t\|f\|_{C^2}(\R), \quad h_1\to 0.
\end{equation}
  Therefore the family
 $T_t^hf$ converges to a family $T_tf$,  as $h \to 0$.
 Clearly the limiting family $T_t$ specifies a strongly continuous semigroup in
 $C_{\infty}(\R)$.

 Applying to $T_t$ the same procedure, as was applied above to $T_t^{\ep}$ (differentiating
 the evolution equation with respect to $x$), shows that $T_t$
 defines
 also a contraction semigroup in $C_{\infty}(\R)\cap C^1(\R)$, preserving positivity of derivatives
 (and hence stochastically monotone) and a bounded semigroup
 in  $C_{\infty}(\R)\cap C^2(\R)$.

Writing
\[
\frac{T_tf-f}{t}=\frac{(T_t-T_t^{\ep})f}{t}+\frac{T_t^{\ep}f-f}{t}
\]
and noting that by \eqref{eqshortdifference} the first term is of
order $o(1) \|f\|_{C^2}$ as $h\to 0$ allows to conclude
that
\[
\lim_{t\to 0} \frac{T_tf-f}{t}=Lf
\]
for any $f\in C_{\infty}(\R)\cap C^2(\R)$. Hence for these $f$, the semigroup $T_tf$ provides classical solutions
to the Cauchy problem $\dot f =Lf$.
By the standard duality argument this implies the required uniqueness.
\end{proof}

Let us discuss a more general situation including unbounded coefficients, where we include a separate
 term in the generator to handle in a unified way a simpler situation of L\'evy measures with a finite first moment.
 Let $C_{\infty,|.|}(\R)$ denote the Banach space of continuous functions $g$ on $\R$ such that $\lim_{x\to \infty} g(x)/|x|=0$,
 equipped with the norm
 $\|g\|_{C_{\infty,|.|}}= \sup_x (|g(x)|/(1+|x|))$.

\begin{theorem}
\label{thonedimstochmonwellposedunboundecoef}
Let for the operator
\[
Lf(x)=\frac{1}{2}G(x) f''(x)+b(x) f'(x)
\]
\begin{equation}
\label{eqgenLKonedimtwojumpparts}
+\int (f(x+y)-f(x)-f'(x)y) \nu (x, dy)
+\int (f(x+y)-f(x)) \mu (x, dy),
\end{equation}
the following conditions hold:

(i) The functions $G(x)$ and $b(x)$ are twice continuously differentiable, $G$ is nonnegative,
 and the first two derivatives of $\nu$ and $\mu$ with respect to $x$ exists weakly
as signed Borel measures
and are continuous in the sense that the integral
\[
 \int f(y) (\nu (x, dy)+ |\nu '(x, dy)|+ |\nu ''(x, dy)|)
 \]
 is bounded and depends continuously on $x$ for any continuous $f(y)\le |y|\land |y|^2$ and the integral
 \[
 \int f(y) (\mu (x, dy)+ |\mu '(x, dy)|+ |\mu ''(x, dy)|)
 \]
 is bounded and depends continuously on $x$ for any continuous $f(y)\le |y|$.

 (ii) For any $a>0$ the functions
 \begin{equation}
\label{eqcondmonotonlevymeasure}
\int_a^{\infty} \nu (x,dy), \quad \int_{-\infty}^{-a} \nu (x,dy)
\end{equation}
are non-decreasing and non-increasing respectively.

(iii) For a constant $c>0$
\begin{equation}
\label{eqcondcritiqu}
\begin{aligned}
& b(x) +  \int |y|(\mu (x, dy) +\int_{-\infty}^{-x} |y+x|\nu (dy)  \le c(1+x), \quad x>1, \\
& -b(x) +  \int |y|(\mu (x, dy) +\int_{-x}^{\infty} |y+x|\nu (dy)  \le c(1+|x|), \quad x<-1.
\end{aligned}
\end{equation}

Then the martingale problem for $L$  in $C^2_c(\R)$ is well
 posed, the corresponding process $X_t^x$ is strong Markov and such that
 \begin{equation}
\label{eqboundforgrowthunboundedcoefLyapunov}
 \E |X_t^x| \le e^{ct}(|x|+c),
 \end{equation}
 its contraction Markov semigroup
 preserves $C(\R)$ and extends from $C(\R)$ to
 a strongly continuous semigroup in $C_{\infty,|.|}(\R)$ with a domain containing $C^2_c(\R)$.
If additionally,
for any $a>0$ the functions
 \begin{equation}
\label{eqcondmonotonlevymeasure}
\int_a^{\infty} \mu (x,dy), \quad \int_{-\infty}^{-a} \mu (x,dy)
\end{equation}
are non-decreasing and non-increasing respectively, then the process $X_t^x$ is stochastically monotone.
\end{theorem}

\begin{proof}
By condition \eqref{eqcondcritiqu} and the method of Lyapunov function (see e.g. Section 5.2 in \cite{Ko10book}) with the
 Lyapunov function $f_L$ being a regularized absolute value, i.e. $f_L(x)$ is twice continuously differentiable positive convex function
 coinciding with $|x|$ for $|x|>1$, the theorem is reduced
to the case of bounded coefficients. And in this case its proof is a straightforward extension of
Theorem \ref{thonedimstochmonwellposed}, where the approximating operator is now
\[
L_hf(x)=\frac{1}{2}G(x) f''(x)+b(x) f'(x)
\]
\[
+\int_{|y|>h} (f(x+y)-f(x)-f'(x)y) \nu (x, dy)+\int_{|y|>h} (f(x+y)-f(x)) \mu (x, dy),
\]
and the rest of the proof remains the same, if one also takes into account that
the diffusion part of this $L$ generates a Feller semigroup due to the well known fact about diffusions with Lipschitz coefficients.
\end{proof}

\section{Processes on the half-line}

\begin{theorem}
\label{thonedimstochmonwellposedunboundecoefhalf1}
Let an operator $L$ be given by \eqref{eqgenLKonedimtwojumpparts} for $x>0$ and the following conditions hold:

 (i) The supports of measures $\nu(x,.)$ and $\mu (x,.)$ belong to $\R_+=\{x>0\}$,
 \[
 \sup_{x\in (0,1]} \left[|b(x)|+G(x)+\int |y| \mu (x, dy)+\int (y\land y^2)\nu (x, dy)\right] <\infty,
 \]
  and the condition (i) of Theorem \ref{thonedimstochmonwellposedunboundecoef} holds for $x>0$. 

 (ii) For any $a>0$ the functions
 \[
\int_a^{\infty} \nu (x,dy), \quad \int_a^{\infty} \mu (x,dy) 
\]
are non-decreasing in $x$.

(iii) For a constant $c>0$
\[
b(x) +  \int y \mu (x, dy) \le c(1+x), \quad x>1. 
\]

Then the stopped martingale problem for $L$ in $C^2_c(\R)$ is well
 posed and specifies a stochastically monotone Markov process $X_t^x$ in $\bar \R_+=\{x\ge 0\}$.
\end{theorem}

\begin{proof}
It follows from Theorem \ref{thonedimstochmonwellposedunboundecoef} and the well known localization procedure (see e.g. \cite{EK}) for
martingale problems. 
\end{proof}

Applying the results from \cite{Ko04} on the boundary points of jump-type processes, one can directly deduce
from Theorem \ref{thonedimstochmonwellposedunboundecoefhalf1} various regularity properties of the stopped process and its semigroup (extending the results from \cite{Ko03} obtained there under
 restrictive technical assumptions). For example, one obtains the following.   

\begin{corollary}
\label{corthonedimstochmonwellposedunboundecoefhalf}
Let the assumptions of Theorem  \ref{thonedimstochmonwellposedunboundecoefhalf1} hold.

(i) Suppose 
\[
G(x)=O(x^2), \quad \int_0^1 z^2 \nu (x,dz)=O(x^2), \quad |b(x)\land 0|=O(x),
\]
for $x\to 0$, then the point $x$ is inaccessible for $X_t^x$ and its semigroup preserves the space $C(\R_+)$ 
of bounded continuous functions on $\R_+$.

(ii) Suppose that $\lim_{x\to 0} b(x)$ exists and
$G(x)=\al x (1+o(1))$ as $x\to 0$ with a constant $\al>0$. If $\al <b(0)$, then
 again the point $x$ is inaccessible for $X_t^x$ and its semigroup preserves the space $C(\R_+)$. If $\al >b(0)$, then
the boundary point $x$ is $t$-regular for $X_t^x$ and its semigroup preserves the space $C(\bar \R_+)$
of bounded continuous functions on $\bar \R_+$.   
\end{corollary}

\end{document}